\documentclass[12pt]{amsart}
\usepackage[utf8]{inputenc}
\usepackage{graphicx}
\usepackage{color}
\usepackage{times}
\usepackage[hidelinks]{hyperref}
\usepackage{enumerate,latexsym}
\usepackage{latexsym}
\usepackage{amsmath,amssymb}
\usepackage{graphicx}
\usepackage{subcaption}
\usepackage{ulem}

\usepackage{amsmath,amsthm,amsfonts,amssymb}
\usepackage{graphicx}
\usepackage{dsfont}

\makeatletter
\def\namedlabel#1#2{\begingroup
 #2%
 \def\@currentlabel{#2}%
 \phantomsection\label{#1}\endgroup
}
\makeatother

\theoremstyle{plain}
\newtheorem*{theorem*}{Theorem}
\newtheorem*{thmex*}{Theorem~\ref{example}}
\newtheorem*{thmasymp*}{Theorem~\ref{thmAsymp}}
\newtheorem{theorem}{Theorem}[section]

\newtheorem{corollary}[theorem]{Corollary}

\newtheorem{proposition}[theorem]{Proposition}

\theoremstyle{definition}
\newtheorem{definition}[theorem]{Definition}

\newcommand{\ben}{\begin{enumerate}}
\newcommand{\een}{\end{enumerate}}

\newcommand{\ed}{\end{document}}

\definecolor{rrr}{rgb}{.9,0,.1}

\definecolor{rr}{rgb}{.8,0,.3}

\usepackage{pdfsync}

\begin{document}

\title[Virtual Multicrossings and Petal Diagrams]{Virtual Multicrossings and Petal Diagrams for Virtual Knots and Links}

\author[Adams et al]{Colin Adams}
\address{Department of Mathematics, Williams College, Williamstown, MA 01267}
\email{cadams@williams.edu}
\author[]{Chaim Even-Zohar}
\address{ Alan Turing Institute, British Library, 96 Euston Road, London, England, NW1 2DB, United Kingdom} \email{chaim@ucdavis.edu}
\author[]{Jonah Greenberg}
\address{325 West End Ave, New York, NY 10023}
\email{jonah325@gmail.com}
\author[]{Reuben Kaufman}
\address{506 West 113th street, Apt. 5a, New York, NY 10025}
\email{reubenkaufman@gmail.com} 
\author[]{David Lee}
\address{Department of Mathematics, Williams College, Williamstown, MA 01267}
\email{djl5@williams.edu}
\author[]{Darin Li}
\address{Department of Mathematics, Williams College, Williamstown, MA 01267}
\email{darinli@gmail.com}
\author[]{Dustin Ping}
\address{1525 Japaul Ln, San Jose, CA 95132}
\email{dping1997@gmail.com.}
\author[]{Theodore Sandstrom} 
\address{Department of Mathematics, Yale University, 
PO Box 208283, New Haven, CT 06520-8283}
\email{theo.sandstrom@yale.edu}
\author[]{Xiwen Wang}
\address{Department of Mathematics, Williams College, Williamstown, MA 01267}
\email{xw2@williams.edu}

\begin{abstract} Multicrossings, which have previously been defined for classical knots and links, are extended to virtual knots and links. In particular, petal diagrams are shown to exist for all virtual knots. 
\end{abstract}

\maketitle

\section{Introduction}

One of the first invariants considered for knots and links was the crossing number $c(L)$, which is the least number of double points for any projection of~$L$, where only two strands can cross at a point. In unpublished work, (Theorem 2.5.2 of \cite{Jones}), Vaughan Jones proved that every knot has a projection with only triple crossings, where instead of two strands crossing, three strands cross, each bisecting the crossing. Using a different method, this was proved independently in \cite{Adams1},  and it was generalized to multicrossings, where $n$ strands cross at a crossing, each bisecting the crossing and each at a distinct height as it passes through the crossing. It was shown that every link has an $n$-crossing projection for $n \geq 2$, meaning that all crossings in the projection are multicrossings with $n$ strands. Hence, one can define $c_n(L)$ to be the least number of $n$-crossings in any $n$-crossing projection of $L$. Thus, there is a spectrum of crossing numbers associated to any given knot or link, starting at the usual crossing number $c_2(L) = c(L)$. These numbers have been further investigated in \cite{Adams3, Adams4, Adams5, Adams2, Adams6}. 

In \cite{kauffman}, Kauffman introduced virtual knots, with projections that also have {\it virtual crossings}, in addition to the usual crossings, called {\it classical crossings}. These projections are equivalent under the usual Reidemeister moves and some additional virtual Reidemeister moves. 

In this paper, we extend the theory of multicrossings to virtual knots and links. In particular, we prove that for every integer $n \geq 2$, every virtual knot or link has a projection consisting of virtual $n$-crossings. 

\begin{definition} A {\it virtual $n$-crossing} is defined as follows. First is an ordering on the integers from 1 to $n$, which corresponds to the heights of the strands passing through a vertical axis read clockwise from the top view, with $1$ being the top one.  By convention we always start with the segment labelled 1.  Second is a subset of the pairings of the integers, which correspond to crossings between strands that are virtual. We further require that if a pair of strands do share a virtual crossing, then any other strand must share a virtual crossing with at least one of them. In other words, no subset of three strands has one virtual and two classical crossings making up the three pairwise double crossings. 
\label{virtualncrossing}
\end{definition}

From the side view, we represent these virtual pairs by arcs connecting the two strands that correspond to a virtual crossing. In Figure \ref{sideview}, we see the side view of a virtual 4-crossing that would be denoted $$\{1243; \, (1,2), (1,3), (2,4), (3,4)\}.$$ We also show one of its possible resolutions, meaning that from the top view, we isotope each segment so they no longer intersect at the central point and any two intersect in a single point not shared by any others. The classical crossings in the resolution are forced to be over or under by the relative heights of the strands.

\begin{figure}[htbp]
    \centering
    \includegraphics[width=.6\linewidth]{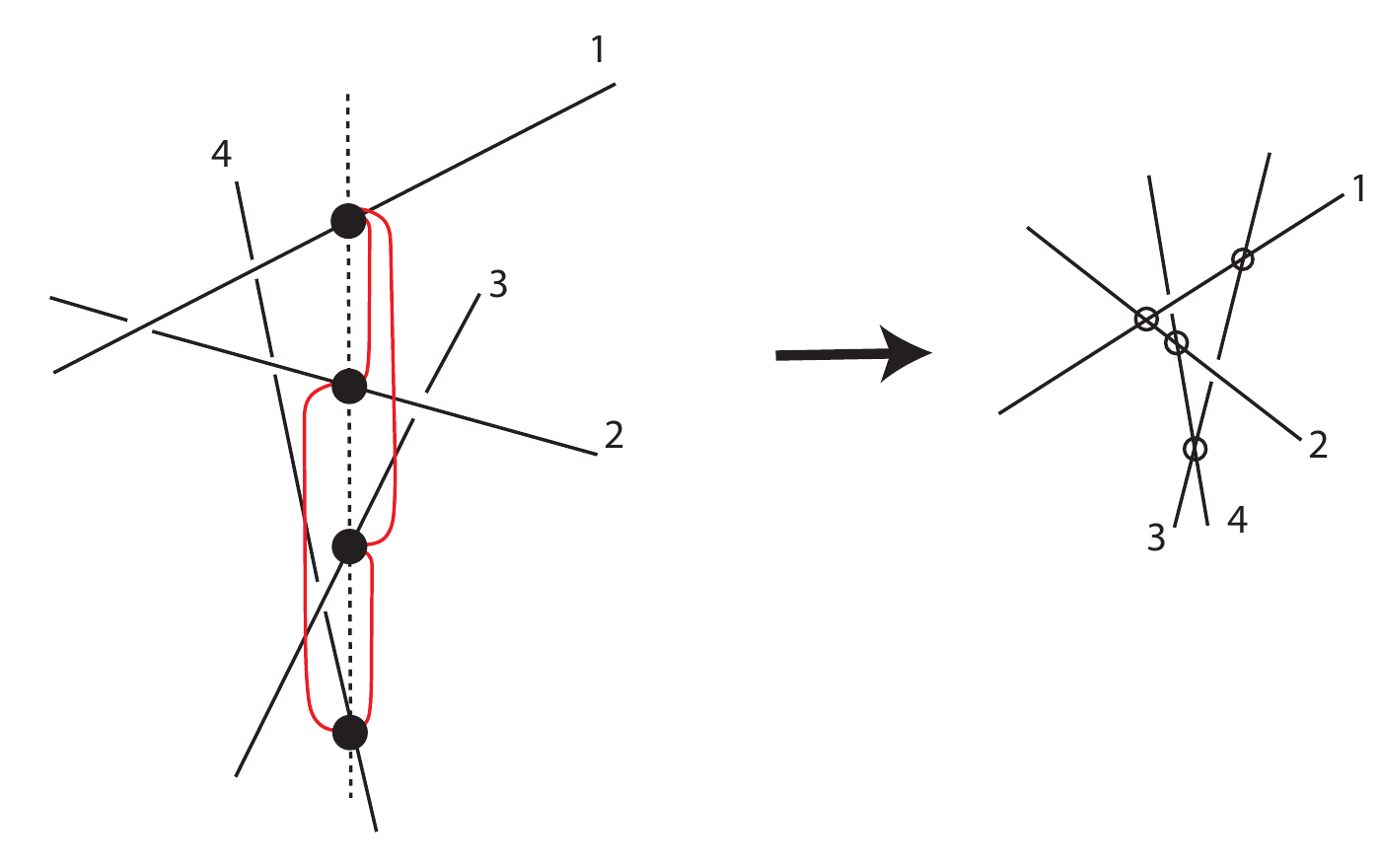}
    \caption{A virtual 4-crossing and one of its resolutions.}
    \label{sideview}
\end{figure}

See Figure \ref{resolutions} for an example of two different resolutions of a given virtual multicrossing. Note that any one resolution must be related to any other by a sequence of valid classical and virtual Reidemeister moves. 

\begin{figure}[htbp]
    \centering
    \includegraphics[width=.6\linewidth]{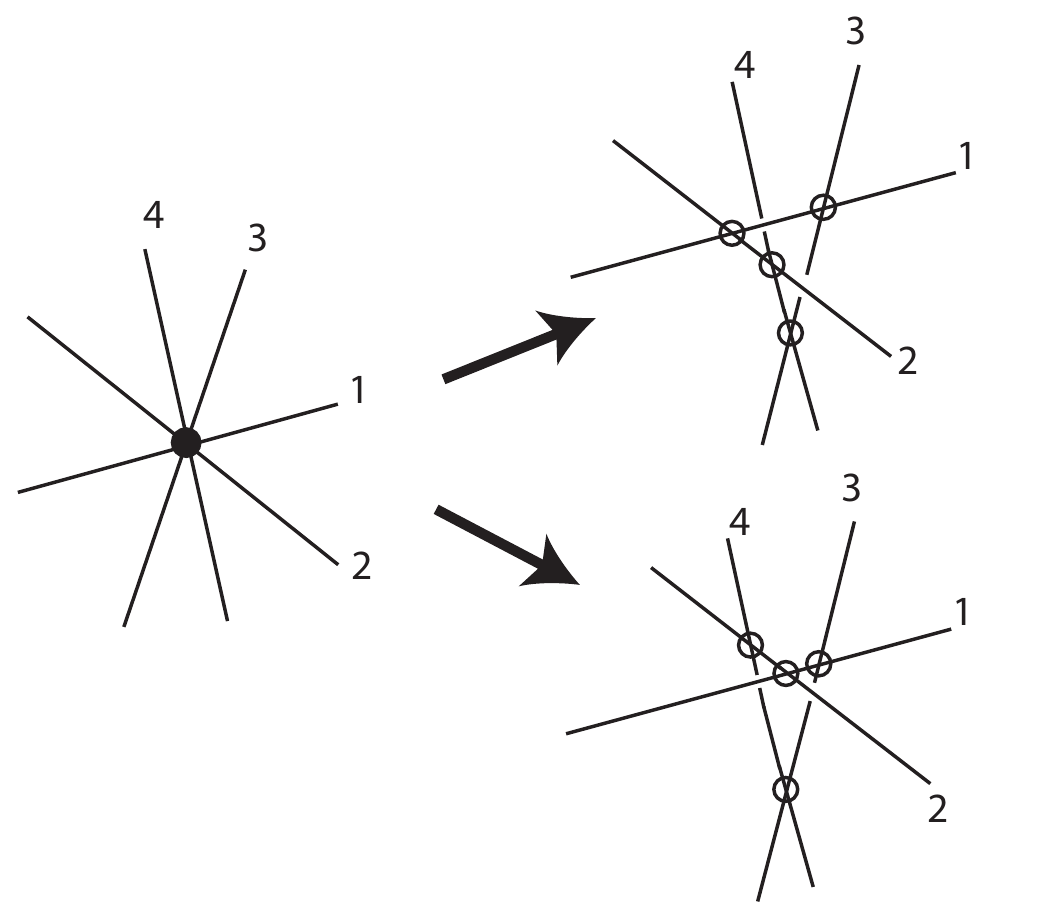}
    \caption{Two different resolutions of the same virtual 4-crossing.}
    \label{resolutions}
\end{figure}

The reason we disallow three strands with two classical and one virtual crossing is that if present when resolving the $n$-crossing into its constituent double crossings, we would have to allow a forbidden move or forbidden detour move as in Figure \ref{forbiddenmove} to get us from one resolution to another. These moves are not allowed in virtual knot theory, as together with the classical and virtual Reidemeister moves, they unknot all virtual knots (cf. \cite{Kanenobu}, \cite{Nelson}, \cite{Ichihara}). See Figure \ref{newforbiddenmove}.

\begin{figure}[htbp]
    \centering
    \includegraphics[width=.6\linewidth]{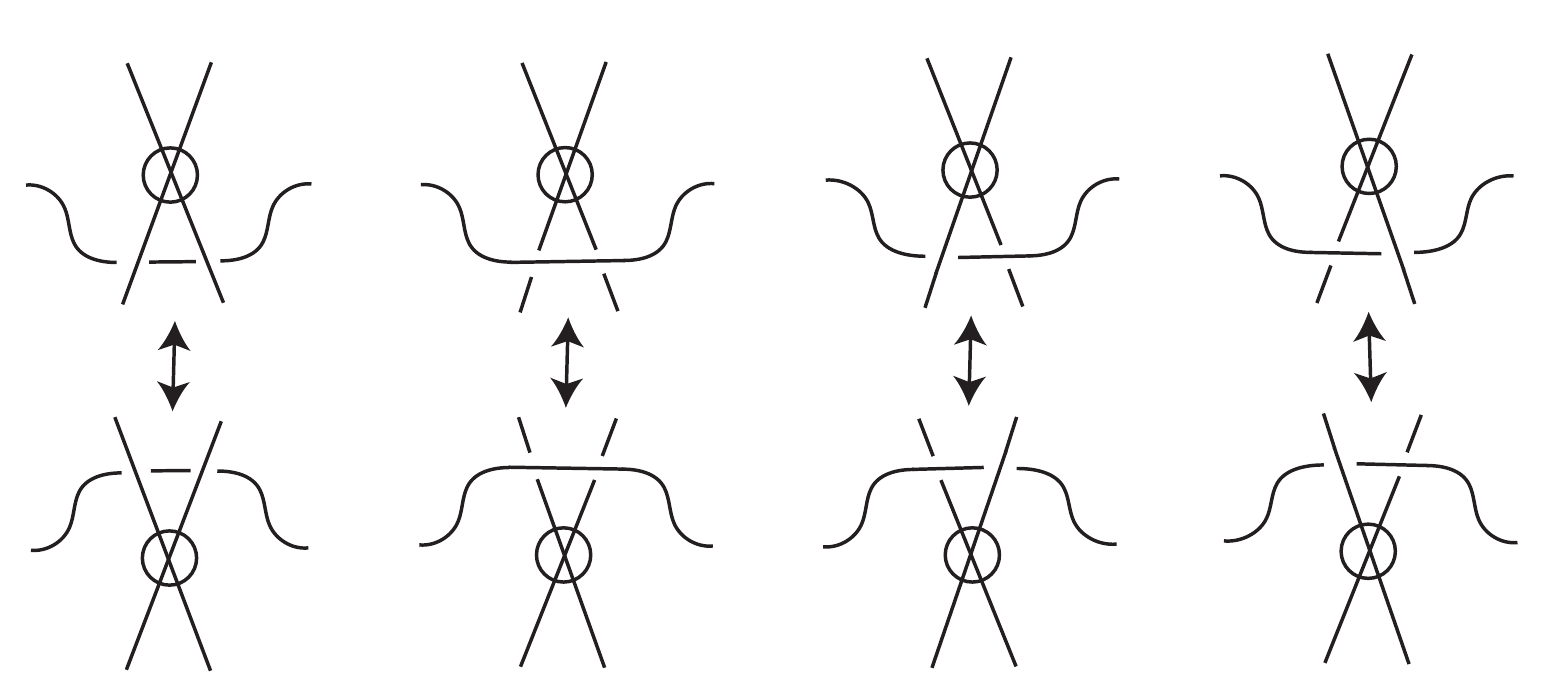}
    \caption{Forbidden moves (left) and forbidden detour moves (right) disallowed in virtual knot theory.}
    \label{forbiddenmove}
\end{figure}


\begin{figure}[htbp]
    \centering
    \includegraphics[width=.6\linewidth]{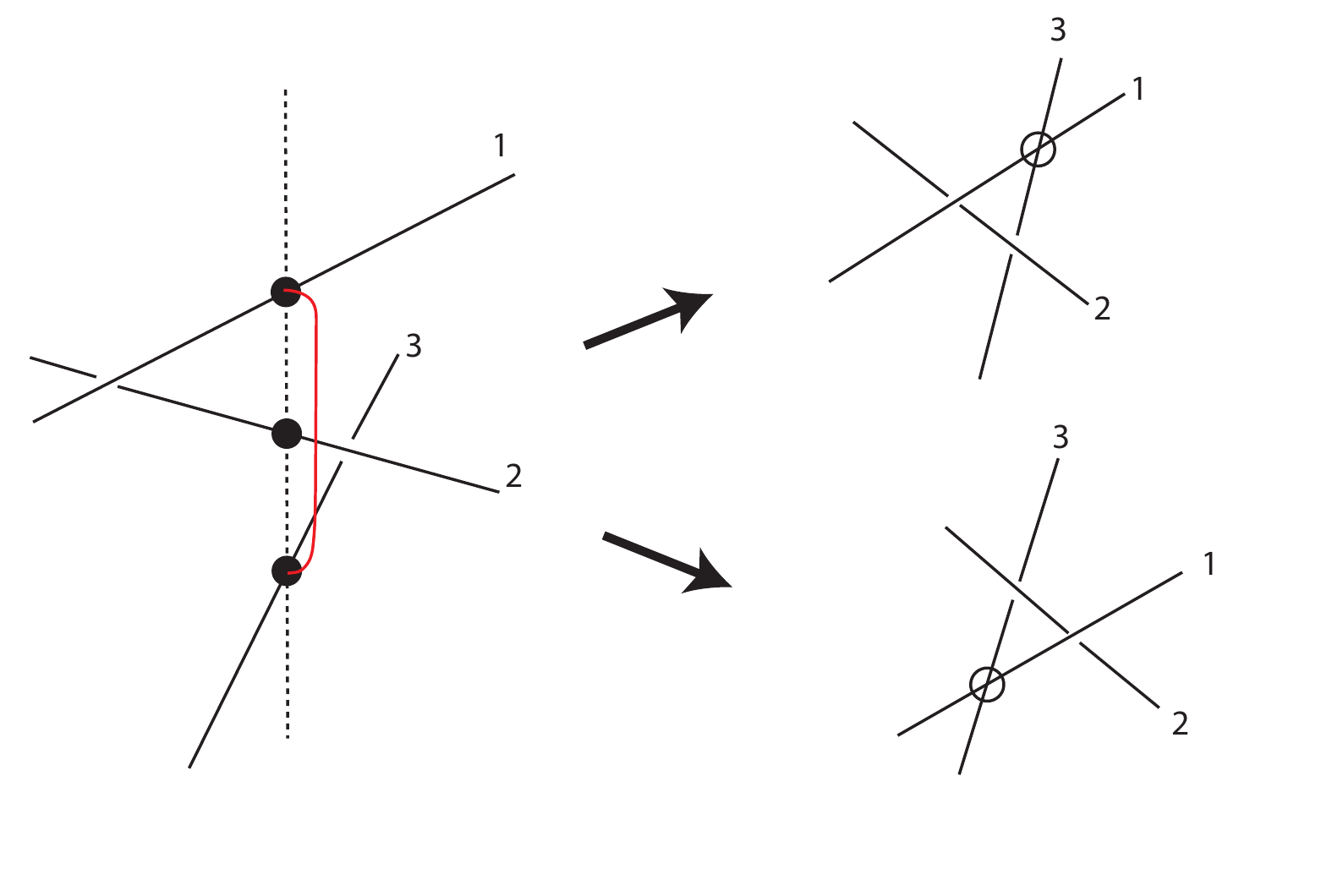}
    \caption{We would need a forbidden detour move to get from one resolution to the other. So this is not a valid virtual 3-crossing.}
    \label{newforbiddenmove}
\end{figure}

In Section 2, we consider virtual triple crossings, and discuss the various types of virtual triple crossings that can be used to obtain every classical and virtual link and the relationships between the various crossing numbers. 

In Section 3, we generalize to show that every link has an almost virtual n-crossing projection and therefore virtual crossing number for all $n \geq 3$.

In Section 4, we investigate the number of different types of virtual $n$-crossings, computing it for small $n$, and finding its asymptotics for large~$n$. 

In Section 5, we consider petal diagrams of virtual knots. A petal diagram of a classical knot is a diagram with a single multicrossing, such that the loops attached to the multicrossing are not nested, as in Figure \ref{petalexample}. The integers denote the relative height that the particular strand passes through the central axis, with the lower integer being the higher height. We show that every virtual knot has a petal diagram with a single valid virtual multicrossing. Therefore, petal number is well-defined for virtual knots.

\begin{figure}[htbp]
    \centering
    \includegraphics[width=.4\linewidth]{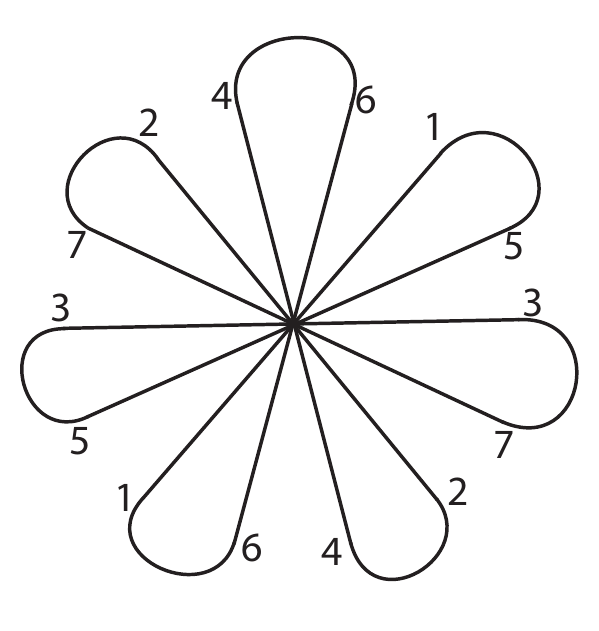}
    \caption{An example of a petal projection of a classical knot.}
    \label{petalexample}
\end{figure}

For notational purposes, given a virtual link $L$,  we define $c_2^{cl}(L)$ to be the least number of classical crossings over all projections, classical and virtual. Define $c_2^{V}(L)$ to be the least number of virtual crossings. And $c_2^T(L)$ is the least total number of crossings of any form over all projections.

Note that in \cite{Manturov}, Manturov proved that for any classical knot or link $L$, $c_2^{cl}(L) = c_2(L)$, which will prove useful.

\section{Virtual Triple Crossing Projections}

We begin by considering the possible types of  virtual triple crossings. We know that every triple crossing resolves into three double crossings, from here on called 2-crossings. Then, in a virtual triple crossing projection, every virtual triple crossing resolves into one of three types: 
 \renewcommand{\theenumi}{\Roman{enumi}}
\begin{enumerate}
    \item 	three classical 2-crossings
\item one classical 2-crossing and two virtual 2-crossings
\item three virtual 2-crossings 
\end{enumerate}

See Figure \ref{Crossing_Types} for examples. As noted previously, we cannot have a virtual triple crossing resolve into two classical 2-crossings and one virtual 2-crossing, since this resolution invokes a forbidden move.

\begin{figure}[htbp]
    \centering
    \includegraphics[width=.6\linewidth]{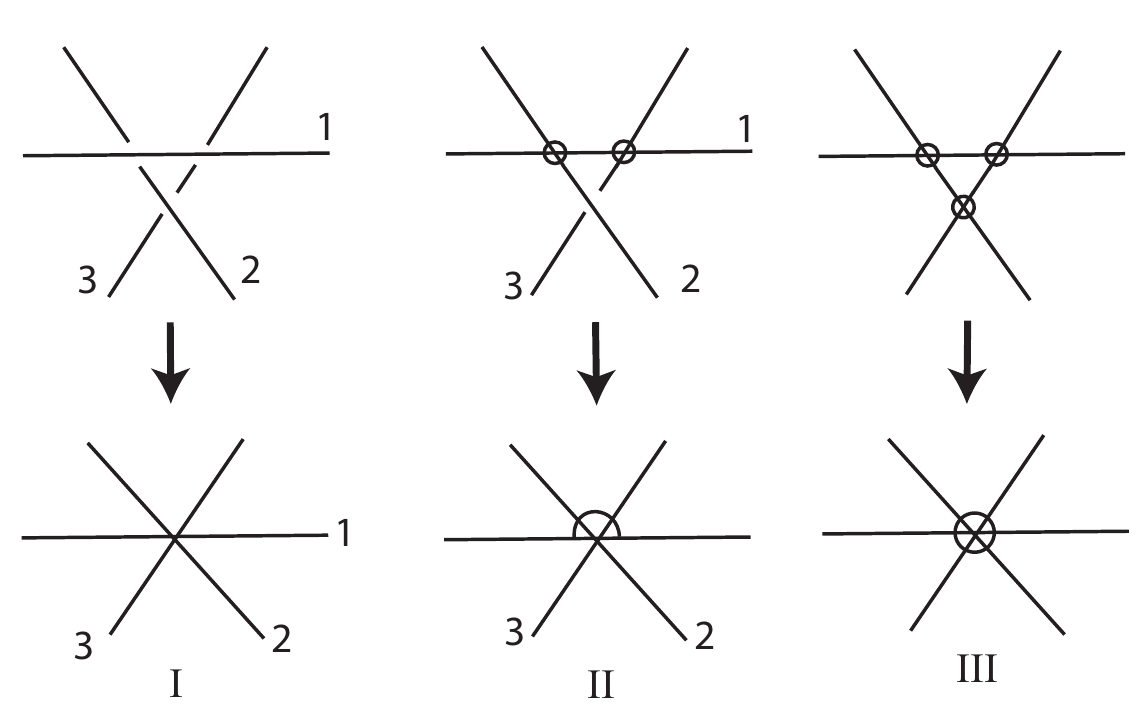}
    \caption{Triple crossing resolutions. Observe that we denote Type II with a half-circle, and Type III with a circle. }
    \label{Crossing_Types}
\end{figure}

Now, we can define four invariants of knots and links related to the types of virtual triple crossings that appear. 

First, we show that every knot or link has a virtual triple crossing projection if we allow Type I, II, and III 3-crossings to appear. We know that every virtual 2-crossing projection $P$ of a knot or a link has a collection of crossing covering circles, which is to say, a collection of disjoint topological circles in the projection plane that only intersect the projection bisecting its crossings such that every crossing intersects a circle. (See \cite{Adams1}).

If a crossing covering circle bisects only classical 2-crossings, and there are no virtual crossings in the disk bounded by the circle, then we choose one of the crossings that it intersects. We can take an over-strand of that crossing and stretch it around the entire crossing covering circle, laying it on top of each of the other crossings intersected by that circle, as shown in Figure~\ref{c3_I}. Note that we eliminate the crossing that we initially chose by stretching its original over-strand over other crossings, ending on the other side of the original under-strand. Therefore, if the crossing covering circle originally intersected $n$ classical double crossings in $P$, then this folding move creates $n-1$ triple crossings. This always works on a classical projection, but cannot be applied if there are any virtual crossings in the disk bounded by a crossing covering circle. 

\begin{figure}[htbp]
    \centering
    \includegraphics[width=.9\linewidth]{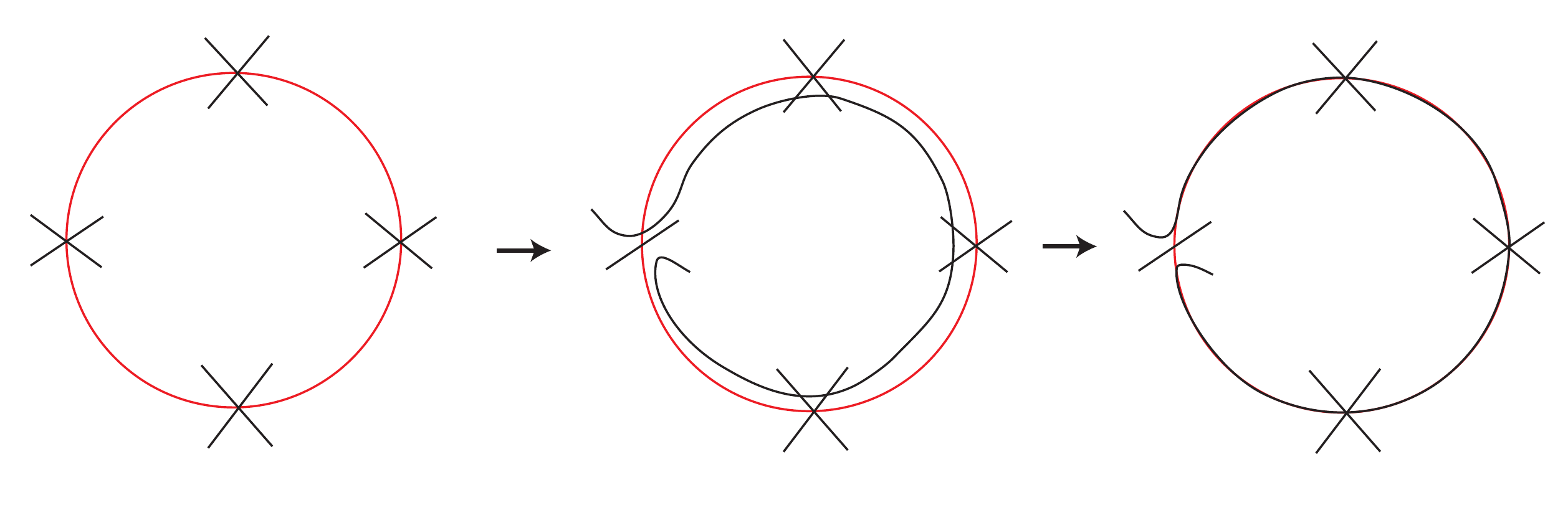}
    \caption{Folding move turning four double crossings into three triple crossings (\cite{Adams1}). }
    \label{c3_I}
\end{figure}

If we perform this folding move over every crossing covering circle in a classical projection $P$, we end up with a triple crossing number that is less than the original double crossing number. 
 
If a crossing covering circle bisects a virtual 2-crossing, then we can fold starting from a virtual 2-crossing, using a detour move (Figure \ref{Detour}). This move changes classical 2-crossings on the circle into Type II 3-crossings, and changes other virtual 2-crossings on the circle into Type III 3-crossings, as in Figure~\ref{c3_II_III}. 

\begin{figure}[htbp]
    \centering
    \includegraphics[width=.7\linewidth]{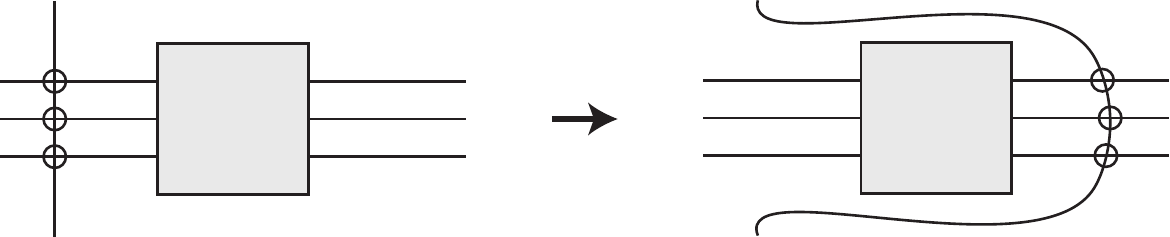}
    \caption{A detour move.}
    \label{Detour}
\end{figure}

\begin{figure}[htbp]
    \centering
    \includegraphics[width=.9\linewidth]{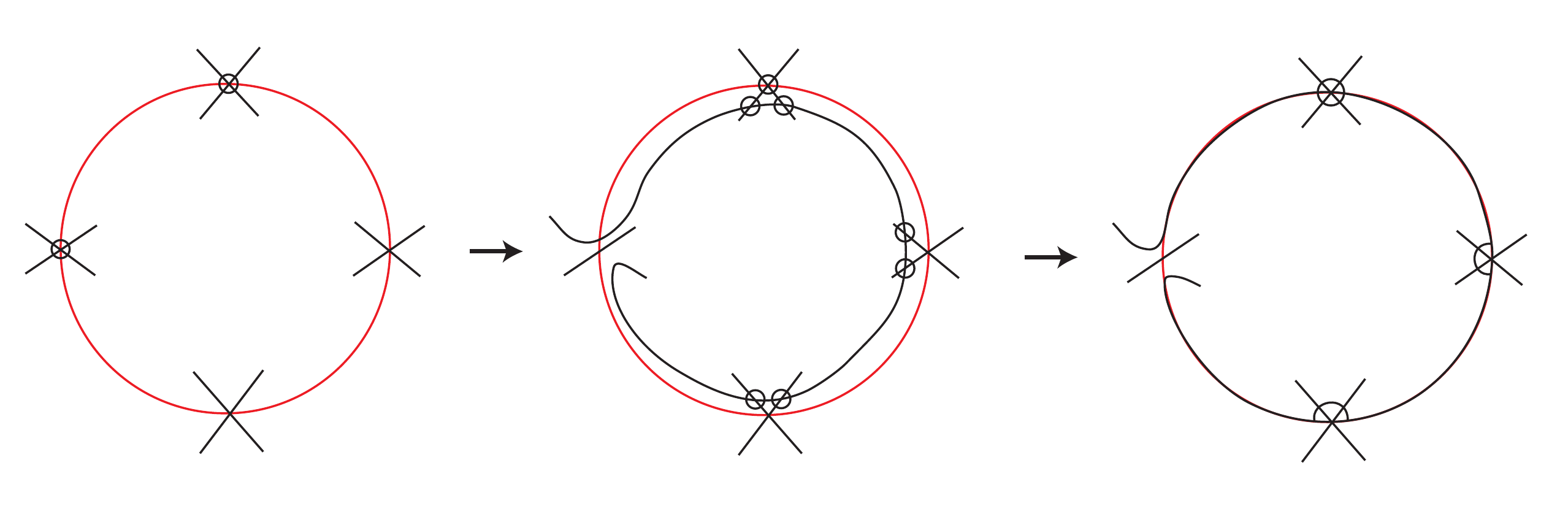}
    \caption{We bring a virtual strand along the interior of the crossing covering circle. Observe the third diagram is the same if we bring the virtual strand along the exterior.} 
    \label{c3_II_III}
\end{figure}

If a crossing covering circle only bisects classical crossings, but either we want to avoid Type I 3-crossings or there are virtual crossings in the disk bounded by the crossing covering circle, we can use the following trick. Using a so-called virtual Reidemeister move IV, create a virtual 2-crossing on a strand near a 2-crossing on the circle, as in Figure \ref{C3_II_III_welldefined}. Then, folding from this virtual 2-crossing will turn the classical 2-crossings into Type II 3-crossings.

Therefore, we can define $c_3^{I,II,III}(L)$ for any knot or link $L$ to be the least number of  3-crossings in any virtual 3-crossing projection of $L$ where Type I, II, III 3-crossings are allowed. This will also be denoted $c_3^T(L)$ as it is the least total number of 3-crossings of any type in a virtual 3-crossing projection of the link~$L$. Note that $c_3^T(L) < c_2^T(L)$ since if the link is classical, this is already known, and if the link is not classical, then in any minimal virtual crossing projection, there will be at least one virtual crossing on a crossing covering circle and the move in Figure \ref{c3_II_III} will decrease the number of crossings. 

Note also that by restricing ourselves to the moves in Figures \ref{c3_II_III} and \ref{C3_II_III_welldefined} we can choose our algorithm to only produce Type II and Type III 3-crossings. Therefore, we can define $c_3^{II,III}(L)$ for any knot or link $L$ to be the least number of virtual 3-crossings in any virtual 3-crossing projection of $L$ where only Type II and III 3-crossings are allowed. (Note: the method of Jones from \cite{Jones} to show every classical knot has a triple crossing projection can easily be modified to prove every virtual knot has a projection with only Type II and Type III virtual crossings.)

\begin{figure}[htbp]
    \centering
    \includegraphics[width=.9\linewidth]{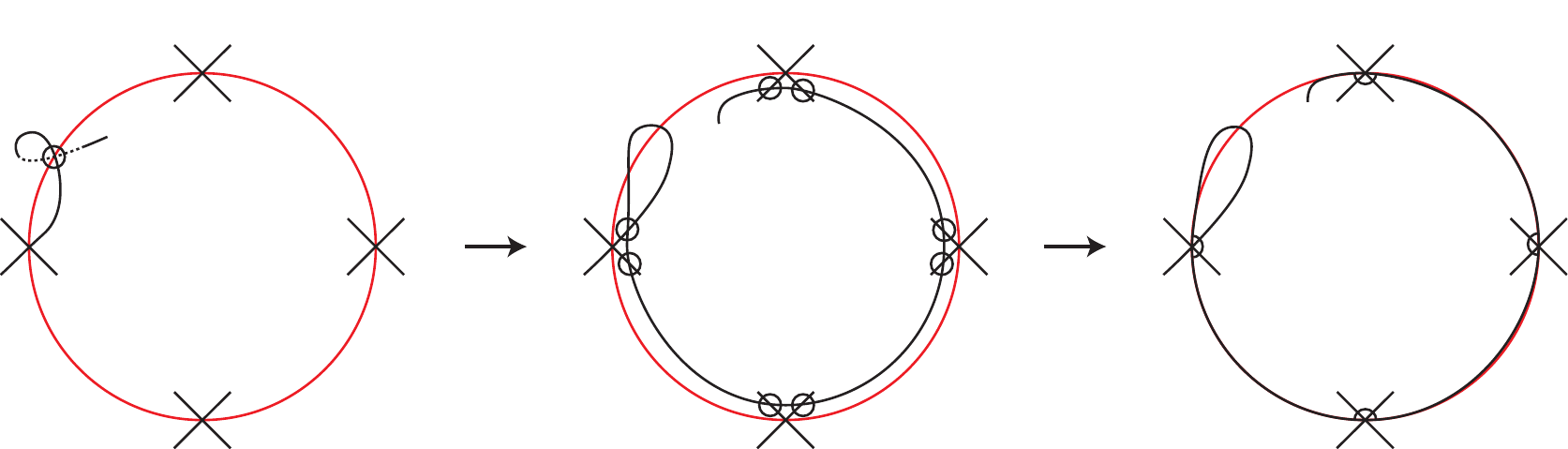}
    \caption{We cut along the dotted line, and detour along the interior of the crossing covering circle.}
    \label{C3_II_III_welldefined}
\end{figure}

It is also true that every knot or link has a virtual triple crossing projection if we restrict to only Type II 3-crossings. We can do the following trick to turn a Type III 3-crossing into Type II 3-crossings. First, resolve a Type III 3-crossing into its three virtual 2-crossings. Next, we perform Reidemeister moves to create three instances of resolved “forbidden 3-crossings,” an area with two classical 2-crossings and one virtual 2-crossing (Figure~\ref{Forbidden_trick}). We fold each “forbidden 3-crossing,” starting from the virtual 2-crossing to create two Type II 3-crossings (Figure~\ref{Forbidden_trick_cont}). Thus, we can turn every Type~III 3-crossing into six Type II 3-crossings. Hence, we can define $c_3^{II}(K)$ and $c_3^{I, II}(K)$  to be the least number of 3-crossings in any virtual 3-crossing projection of $K$ where only Type II 3-crossings are allowed or where only Type~I and II 3-crossings are allowed.

\begin{figure}[htbp]
    \centering
    \includegraphics[width=.5\linewidth]{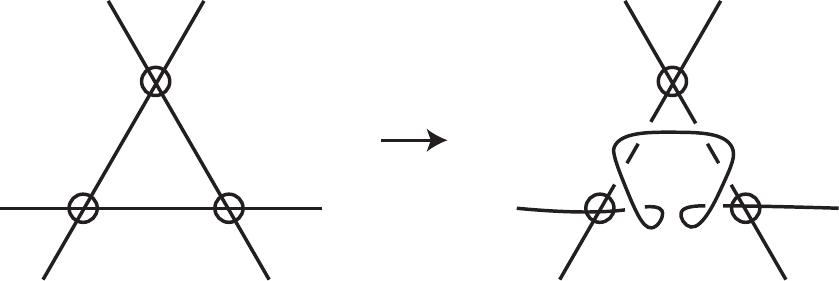}
    \caption{The third diagram enlarges the three instances of a resolved “forbidden 3-crossing.”}
    \label{Forbidden_trick}
\end{figure}

\begin{figure}[htbp]
    \centering
    \includegraphics[width=.9\linewidth]{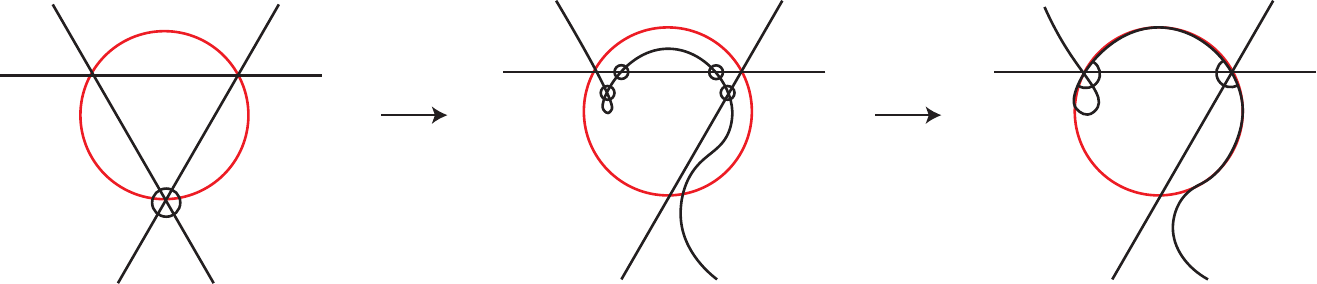}
    \caption{Turning a forbidden 3-crossing into two Type~II 3-crossings.}
    \label{Forbidden_trick_cont}
\end{figure} 

The following relations between triple crossing numbers of virtual links are immediate from the definitions:
$$c_3^T(L) \leq c_3^{II,III}(L) \leq c_3^{II}(L)$$
$$c_3^T(L) \leq c_3^{I,II}(L) \leq c_3^{II}(L)$$

\begin{theorem} For a nontrivial classical knot or link $L$, $c_3^T (L) < c_3^{II,III}(L) =c_3^{II}(L)  = c_2(L)$.
\end{theorem}

\begin{proof} We have already explained why $c_3^T(L)$ is smaller that $c_2(L)$. We next show $c_3^{II}(L)=c_2(L)$. With the trick that turns every classical 2-crossing into a Type II 3-crossing (see Fig.~\ref{C3_II_III_welldefined}), we can change a 2-crossing projection of a classical link with $c_2$  crossings to obtain a virtual triple crossing projection with the same number of $c_2$ Type II 3-crossings. So, $c_3^{II}(L) \leq c_2(L)$. But every Type II 3-crossing, resolved, has one classical 2-crossing, so the projection of a link L with $c_3^{II}$ Type II 3-crossings has $c_3^{II}$  classical 2-crossings when resolved. Note that this resolved projection of $L$ is a virtual projection. We know from \cite{Manturov} that for a classical link, the least number of classical crossings over all projections, with or without virtual crossings, must occur in a classical projection. So, $c_2(L) \leq c_3^{II}(L)$.

Similarly, we show $c_3^{II,III}(L) = c_2(L)$ for classical knots and links. On the one hand, a virtual triple crossing projection with only Type~II and Type~III 3-crossings, when resolved, has at most $c_3^{II,III}$ classical 2-crossings. But then by \cite{Manturov}, $c_2(L) \leq c_3^{II,III}(L)$. On the other hand, $c_3^{II,III}(L) \leq c_3^{II}(L)=c_2(L)$ for classical knots and links.
\end{proof}

Now let us find these new invariants for some simple examples. \medskip
\renewcommand{\theenumi}{\roman{enumi}}
\begin{enumerate}
\item By enumerating ways to connect the ends of a virtual triple crossing without creating more crossings, we find knots or links $L$ with $c_3^T(L)$ or $c_3^{I,II}(L)$ equal to~1. The only nontrivial such links are the classical and virtual Hopf link, the latter being a Hopf link where one classical crossing has been made virtual. The only link with $c_3^{II,III}(L)$ or $c_3^{II}(L)$ equal to~1 is the virtual Hopf link. So, 
$$ c_3^T(L) =c_3^{II,III}(L) = c_3^{I,II}(L) =c_3^{II}(L) =1 $$ 
for the virtual Hopf link, and these invariants are all at least 2 for every nontrivial knot or link other than the Hopf link and virtual Hopf link.

\item Thus, all the virtual 3-crossing numbers of the classical trefoil knot $3_1$ are at least~2. Note that since the standard projection of the trefoil has one crossing covering circle, folding from one 2-crossing will give a projection with two Type I 3-crossings. So 
$$ c_3^T(3_1)=c_3^{I,II}(3_1)=2 $$ 
Since $3_1$ is a classical knot, 
$$ c_3^{II}(3_1)= c_3^{II,III}(3_1)=c_2(3_1)=3 $$ 
	
\item For the classical figure-8 knot $4_1$, folding along two crossing covering circles gives 
$$ c_3^T (4_1 )=c_3^{I,II} (4_1 )=2 $$ 
Since it is classical, 
$$ c_3^{II} (4_1 )=c_3^{II,III} (4_1 )=c_2(4_1 )=4 $$
	
\item For the virtual trefoil $VT$, detouring from the virtual 2-crossing as in Figure~\ref{Trefoil} gives two Type~II 3-crossings. Therefore, $$ c_3^T (VT)=c_3^{II,III} (VT)=c_3^{I,II} (VT) = c_3^{II}(VT)=2 $$
\end{enumerate}
 
\begin{figure}[htbp]
    \centering
    \includegraphics[width=.7\linewidth]{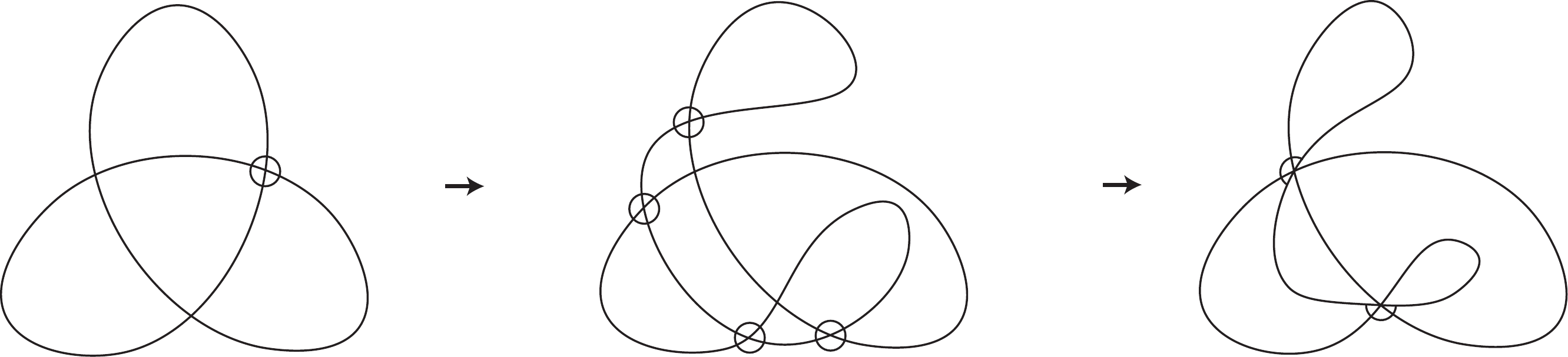}
    \caption{A virtual detour and a diagarm with two triple crossings for the virtual trefoil.}
    \label{Trefoil}
\end{figure}

As we have seen,  for any nontrivial classical link $L$, $$c_3^T(L) < c_2(L) = c_3^{II,III}(L) =c_3^{II}(L)$$ by folding along a covering collection of circles of a minimal 2-crossing projection. In fact, if $L$ is not a 2-braid link, then $c_3^T(L) \leq c_2(K) - 2$. (See \cite{Adams1}).\\

\noindent {\bf Question 1:} We know that for any classical link,  $c_2^T(L) = c_2(L)$. Is it also true that $c_3^T(L) = c_3(L)$? \\

\noindent {\bf Question 2:} Where do $c_3^{I,II}(L)$ and $c_3^{II,III}(L)$ fit in between  $c_3^T(L)$ and $c_3^{II}(L)$? We have examples where $c_3^{I,II}(L) < c_3^{II,III}(L)$. Do there exist examples where $c_3^{I,II}(L) > c_3^{II,III}(L)$?\\

\noindent {\bf Question 3:} Is $c_3^{I, III}(L)$  well-defined?

\section{Almost Virtual Multicrossings}
We can generalize the concept of a type II triple crossing to multicrossings of higher order.
\begin{definition}
An $n$-crossing is {\it almost virtual} if exactly one of its constituent 2-crossings is classical. A projection of a link is said to be an {\it $n$-almost-virtual} projection if all of its crossings are almost virtual $n$-crossings.
The $n$-almost-virtual crossing number of a link $L$, written $c_n^{AV}(L)$, is the least number of almost virtual $n$-crossings in any $n$-almost-virtual projection of $L$.
\end{definition}

To show that $c_n^{AV}(L)$ is well defined over all links, we provide an algorithm to generate an $n$-almost-virtual projection from any 2-projection. We use a generalized version of the approach from \cite{Adams6}. Note that any almost virtual $3$-crossing can be turned into an almost virtual $5$-crossing by the detour move shown in Figure \ref{fig:av-zig-zag}. A similar trick works to convert any almost virtual $n$-crossing into an almost virtual $(n+2)$-crossing. Since every link has a projection consisting of only type II triple crossings, which is an almost virtual 3-crossing projection, all links have almost virtual $n$-crossing projections for all odd $n$. By the series of Reidemeister moves depicted in Figure \ref{fig:2-to-4-av}, we can convert any 2-crossing projection into a 4-almost-virtual projection. Therefore, every link has an $n$-almost-virtual projection for all nonnegative $n$.

\begin{figure}[htbp]
  \centering
  \includegraphics[scale=.7]{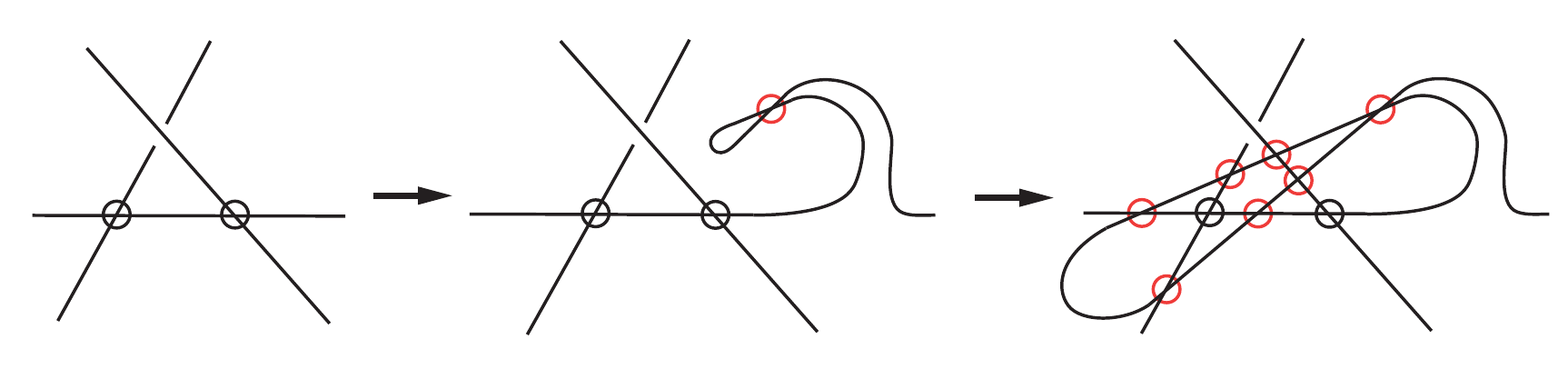}
  \caption{Transforming an almost virtual $3$-crossing into an almost virtual $5$-crossing by a detour move.}
  \label{fig:av-zig-zag}
\end{figure}

\begin{figure}[htbp]
  \begin{subfigure}{1.0\linewidth}
  \center
    \includegraphics[scale=.8]{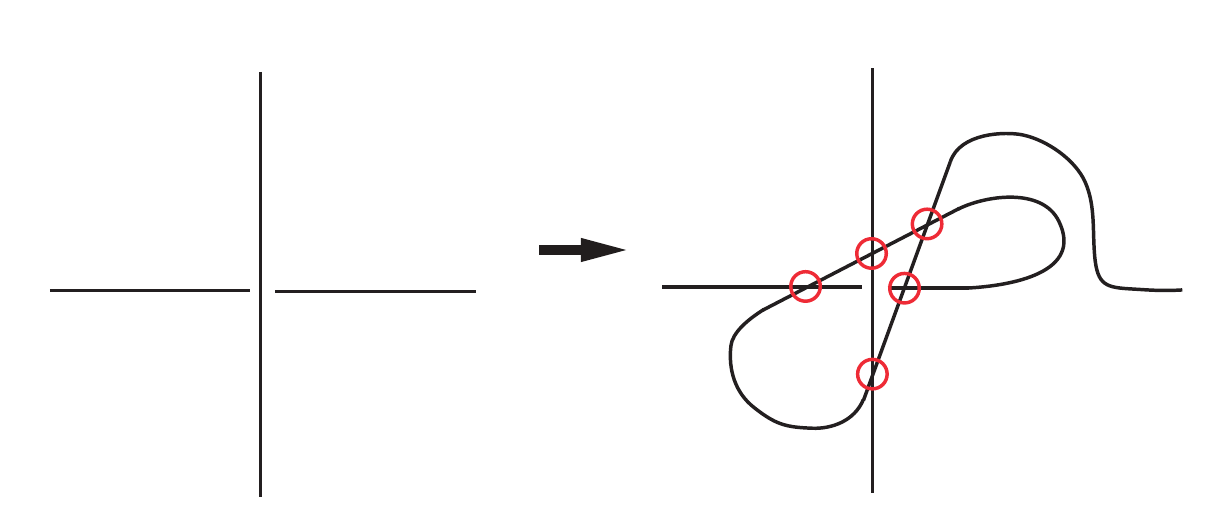}
    \caption{Classical crossing}
  \end{subfigure}
  \hspace*{\fill}
  \begin{subfigure}{1.0\linewidth}
  \center
    \includegraphics[scale=.7]{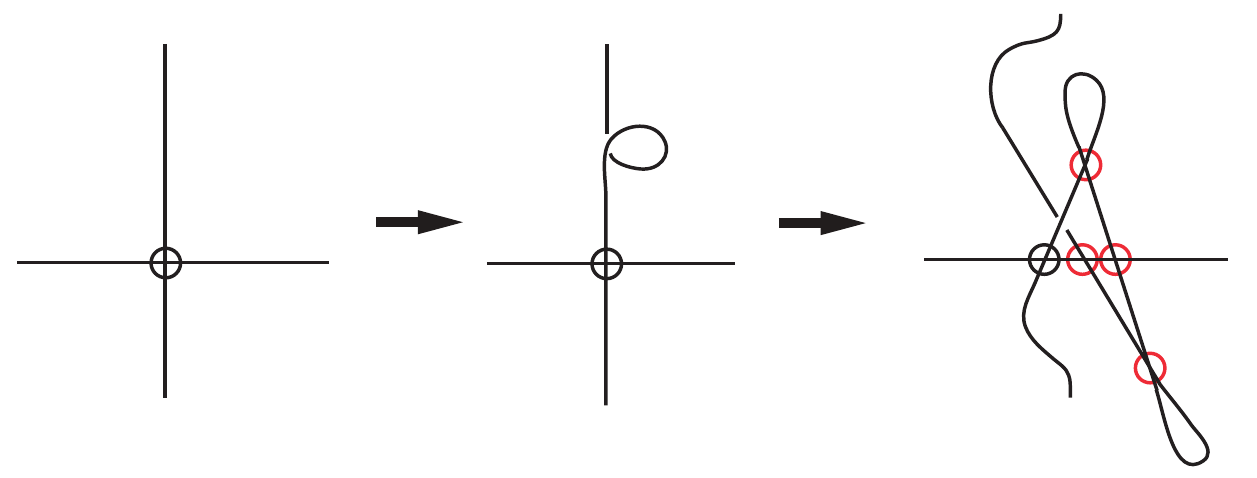}
    \caption{Virtual crossing}
  \end{subfigure}
  \caption{Transforming a 2-crossing into an almost virtual 4-crossing.}
  \label{fig:2-to-4-av}
\end{figure}

For the same reason that $c_3^{II}(L) = c_2(L)$ for any classical link, we have $c_n^{AV}(L) = c_2(L)$ for any classical link $L$ and all integers $n \geq 2$.

Note that not all almost virtual $n$-crossings are the same. For example, the classical 2-crossing in an almost virtual 4-crossing may be either shared between two adjacent strands or two nonadjacent strands. In terms of Definition~\ref{virtualncrossing}, an almost virtual crossing can be described by any of the $n!$ possible orderings for the heights of the strands and one of the $\tbinom{n}{2}$ pairs for the classical crossing. However, looking on the resolutions of these multicrossings, we obtain only $n-1$ nonequivalent almost virtual $n$-crossings up to rotation, depending on the clockwise distance from the over crossing to the under crossing in the classical pair. Similarly, up to rotation and reflection, there are $\lfloor \frac{n}{2} \rfloor$ different almost virtual $n$-crossings for all $n \geq 2$. \\


\noindent {\bf Question 4:} Generalizing Question 1, is it true for classical links that $c_n^T(L) = c_n(L)$? 

\section{Counting Virtual \textit{n}-Crossings}


We have seen that a multicrossing projection of a virtual knot may have fewer crossings than a double crossing one. On the other hand, there are more possible types of crossings. It is useful to consider how many different virtual $n$-crossings there are.

Consider a virtual multicrossing of $n$ straight arcs, labeled $\{\alpha_1,\dots,\alpha_n\}$, say clockwise. When the multicrossing is resolved, some pairs of arcs cross virtually and some classically, and we first explore the different ways this can happen.

Define a symmetric binary relation on the arcs: $\alpha_i \sim \alpha_j$ iff either $\alpha_i$ and $\alpha_j$ cross classically or $\alpha_i = \alpha_j$. The condition in Definition~\ref{virtualncrossing} that the virtual multicrossing is valid is now equivalent to the transitivity of this relation. Indeed, a forbidden configuration arises exactly when $\alpha_i \sim \alpha_j$ and $\alpha_j \sim \alpha_k$ but $\alpha_i \not\sim \alpha_k$. The validity condition hence says that ``$\sim$'' is an equivalence relation. Such a relation is uniquely defined by the partition of the set $\{\alpha_1,\dots,\alpha_n\}$ into equivalence classes, which are maximal subsets of arcs that pairwise cross classically.

The number of such \textit{set partitions} of $n$ elements is the well-studied $n$th \textit{Bell number}, denoted~$B_n$. This sequence has a super-exponential asymptotic growth: $B_n$ is roughly $n!/(e \log n)^n$, up to a sub-exponential factor as $n \to \infty$~\cite[p.~562]{flajolet2009analytic}. See OEIS~\cite[A000110]{oeis} for more information on these numbers. 

In order to describe the virtual multicrossing in full, we also have to specify the relative heights of the arcs. These are needed to determine whether one arc crosses over or under another one, whenever they are in the same equivalence class. Therefore, the additional information is precisely a linear ordering of the elements of each part in the partition. Following~\cite[p.~563]{flajolet2009analytic}, we refer to partitions into ordered parts as \textit{fragmented permutations}. Their number is given by~$F_n = \sum_{k=1}^{n}\tbinom{n-1}{k-1}\frac{n!}{k!}$, where the $k$th term corresponds to partitions into $k$ parts. See also OEIS~\cite[A000262]{oeis}.

Finally, we may want to identify virtual multicrossings that are rotations of one another, differing only by the arbitrary choice of the arc labeled by $\alpha_1$. Formulas that count set partitions up to rotation are given in~\cite[A084423]{oeis}. Similarly, \cite[A084708]{oeis} also takes reflections into account. However, we are interested in the corresponding quantities for fragmented permutations of $\{\alpha_1,\dots,\alpha_n\}$, where the parts are ordered. We denote the desired count by~$V_n$.

It follows from the above discussion that $F_n$ is an overcount. Counting with labels, the same multicrossing might be counted $n$ times. This happens when all $n$ rotations of the multicrossing are different. Some crossings are counted fewer times, such as the completely virtual case, counted once. This already means that $F_n/n \leq V_n \leq F_n$, not a dramatic variation considering the trivial lower bound $F_n \geq n!$ for example.

If $n$ is prime then the only two possibilities are the above-mentioned cases, overcount by $n$ and all-virtual, as no other multicrossing has a rotational symmetry. 
This leads to the following formula for prime $p$:
$$ V_p \;=\; \frac{F_p-1}{p}+1 $$

For example, using $F_2 = 3$ we obtain $V_2 = 2$ types of 2-crossings, as we know, virtual and classical. Similarly $F_3 = 13$ yields $V_3=5$ types as in Figure \ref{fig:possibletriples}, which are: two classical 3-crossing of Type I in our terminology, two mixed of Type II, and one purely virtual of Type III. Note that the two different triple crossings in each of Types I and II are mirror images of each other.

 \begin{figure}[htbp]
  \centering
  \includegraphics[scale=.6]{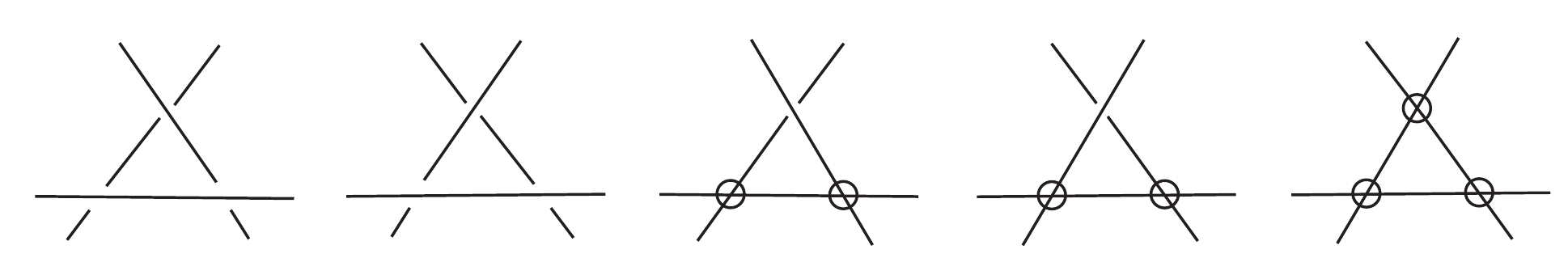}
  \caption{The five possible virtual triple crossings up to rotation.}
  \label{fig:possibletriples}
\end{figure}

For composite $n$, there are more kinds of overcount. For example, consider the virtual 4-crossing where $\alpha_1$ goes over $\alpha_2$, $\alpha_3$ over $\alpha_4$, and the remaining crossings are virtual. This is readily invariant under clockwise rotation forward by two arcs, but not by one. Thus it is counted twice in $F_4$, not four times, neither once.

The remedy for such complications is Burnside's Lemma for counting orbits of a group action. In our case, it says $V_n = \tfrac{1}{n}\sum_{r=1}^{n} \mathrm{Fix}(n,r)$, where $\mathrm{Fix}(n,r)$ is the number of labeled virtual $n$-crossings invariant under rotation by $r$ arcs. Replacing $r$ by $d=\mathrm{gcd}(n,r)$, this sum can be rewritten as $\tfrac{1}{n}\sum_{d|n} \varphi(\tfrac{n}{d})\mathrm{Fix}(n,d)$, where $\varphi$ is Euler's totient function. 

In a labeled virtual $n$-multicrossing invariant under rotation by $d|n$, the crossings between the arcs in a congruence class $\{\alpha_i,\alpha_{i+d},\alpha_{i+2d},\dots\}$ must be all virtual. Indeed, the rotations generated by $d$ act transitively on these arcs, and no crossing with classically ordered arcs is invariant under all these rotations. A valid invariant multicrossing can hence be described by picking one representative from every congruence class, and a fragmented permutation that determines the relations between the $d$ arcs, and then all classically crossing pairs are generated by their $\tfrac{n}{d}$ rotations. This gives $(\tfrac{n}{d})^dF_d$ ways to describe labeled virtual $n$-crossings invariant under rotations by~$d$. However, shifting the representatives in a whole ordered part together by~$d$ does nothing, so we have to divide by $(\tfrac{n}{d})^k$ where $k$ is the number of parts in the fragmented permutation. Using Burnside's Lemma and the refined sum for~$F_d$ mentioned above, the exact enumeration is as follows.

\begin{proposition}
The number of types of virtual $n$-crossings, considered up to rotation, is
$$ V_n \;=\; \frac{1}{n} \sum\limits_{d|n} \varphi\left(\frac{n}{d}\right) \sum\limits_{k=1}^{d} \binom{d-1}{k-1} \frac{d!}{k!} \left(\frac{n}{d}\right)^{d-k} $$
\end{proposition}


\begin{corollary}
For $n=2,3,\dots,10$, the number of virtual $n$-crossings up to rotation is respectively $$ 2,\; 5,\; 20,\; 101,\; 684,\; 5377,\; 49342,\; 510745,\; 5894550 $$
\end{corollary}


Finally, we derive an asymptotic answer. It is based on the observation that the term $d=n$ in the above sum is overwhelmingly more dominant than the other terms, of $d \leq n/2$. This essentially reflects the fact that any rotational symmetries in a typical large $n$-crossing become rare.

In more detail, the inner sum over $\tbinom{d-1}{k-1}(d/n)^k/k!$ is clearly increasing in~$d$, and the remaining factor $({n}/{d})^dd!$ contains $d$ of the $n$ terms in the product~$n!$, so its value for $d \leq n/2$ is super-exponentially smaller than for $d=n$. Thus $V_n$ is asymptotically equivalent to $F_n/n$ also in the composite case. Dividing by $n$ the estimate of $F_n$ from  \cite[p.~563]{flajolet2009analytic} yields the following

\begin{corollary}
The number of  virtual $n$-crossings up to rotation is
$$ V_n \;\sim\; \frac{n! \, e^{2\sqrt{n}}}{2\sqrt{\pi e} \, n^{7/4}} $$
in the sense that the ratio between the two sides of the equation converges to $1$ as $n \to \infty$.
\end{corollary}

\section{Petal Diagrams of Virtual Knots}
 


Given a classical knot, it was proved in \cite{Adams2} that it has a petal diagram,  as in Figure \ref{petalexample}. Each line segment is assigned a height with the top labelled~1 and the bottom labelled $n$, which then determines the crossings in any resolution of the single multi-crossing. We extend that result to virtual knots here, by putting a virtual $n$-crossing at the center of the petal diagram.

\begin{theorem} Any virtual knot has a virtual petal diagram.
\end{theorem}

\begin{proof} We show that there exists an algorithm for putting a given virtual knot into a diagram with a single valid virtual multi-crossing at the center and no nested petals. To do so, we use signed Gauss codes, as described in ~\cite[Section 3.2]{kauffman}. Any virtual knot diagram can be delineated by a signed Gauss code obtained as follows. Given a knot diagram $K$ with $n$ classical crossings, pick an orientation for the knot, and then start on the projection at any given point, and travel in the direction of the orientation. At each classical crossing that is encountered, give it a new integer if it is new, or the same integer it was previously given if we have encountered it before. Record that number. This yields a list of $2n$ integers. If we are traversing the crossing on an overstrand, precede the digit with an O, and if on an understrand, precede it with a U. Finally for each digit, add a $+$ or $-$ after the digit depending on whether the crossing is a $+$ crossing or a $-$ crossing, see Figure~\ref{crossingsigns} below. For each virtual crossing, we ignore it. The resulting signed Gauss code can be shown to correspond to a unique virtual knot. For example, O1+U2+U1+O2+ is a virtual trefoil. We will show that given any valid Gauss code, there is a petal diagram that corresponds to it, and therefore every virtual knot has a petal diagram.

Given a signed Gauss code of length $2n$, which would correspond to a virtual knot diagram with $n$ classical crossings, we construct  a petal diagram for it with a central $3n$-crossing if $n$ is odd and a central $(3n+1)$-crossing if $n$ is even.  We orient the petal projection so that we move clockwise around it, and choose a starting point represented by a black dot on the outside of a petal. Moving clockwise from that point, we designate the next line segment through the center to be labelled 1 and all subsequent line segments are denoted by the subsequent integers. Write the list of line segments in order beneath the Gauss code, associating one line segment with each first appearance of a crossing number and two line segments with each second appearance of a crossing number. So for instance we have a correspondence as seen in Table \ref{tab:Gausscodetable}.

\begin{table}[htbp]
\centering
\begin{tabular}{c|c|c|c|c|c|c}
Gauss code &  O1+ & U2+ & U1+ & O3- & O2+ & U3-\\
Line segment    & 1& 2& 3 4 & 5  & 6 7 & 8 9 \\
\end{tabular}
\\[0.5em] \phantom{...}
\caption{Constructing a petal diagram for a virtual knot diagram using the signed Gauss code.}
\label{tab:Gausscodetable}
\end{table}

From this, we construct a resolution of our petal diagram. Given a particular classical crossing, we consider the first segment associated to its appearance in the code. The two segments associated with its second appearance will cross the first segment one right to left and one left to right. One of those will be in alignment with the over/under and sign for the crossing, as in Figure \ref{crossingsigns}. 

\begin{figure}[htbp]
\centering
\includegraphics[scale=.7]{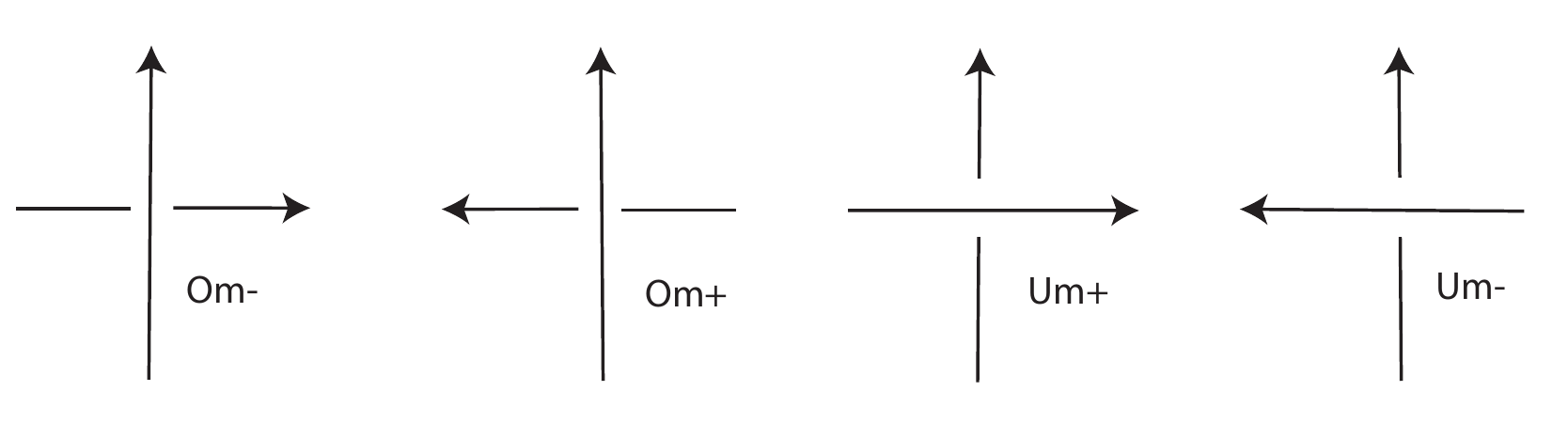}
\caption{We must respect the directions on the segments at a crossing labelled $m$.}
\label{crossingsigns}
\end{figure}

Put in a classical crossing between the first segment and the one whose direction respects the U/O and +/- labels. For each of the two crossing segments, make all other places that they cross another segment into virtual crossings. And for the segments we did not use, make all their crossings virtual. Since any two segments cross classically at most once, it is straightforward to choose heights for the line segments so that the classical crossings reflect the U's and O's. 

By the time we are done, we have a diagram of a virtual knot with signed Gauss code that matches the code with which we started. See Figure \ref{petalconstruction} for the example from Table \ref{tab:Gausscodetable}. Because no segment has more than one classical crossing, this satisfies the condition for the resulting virtual multi-crossing that no three segments generate two classical and one virtual crossing.
\end{proof}

\begin{figure}[htbp]
\centering
\includegraphics[scale=.5]{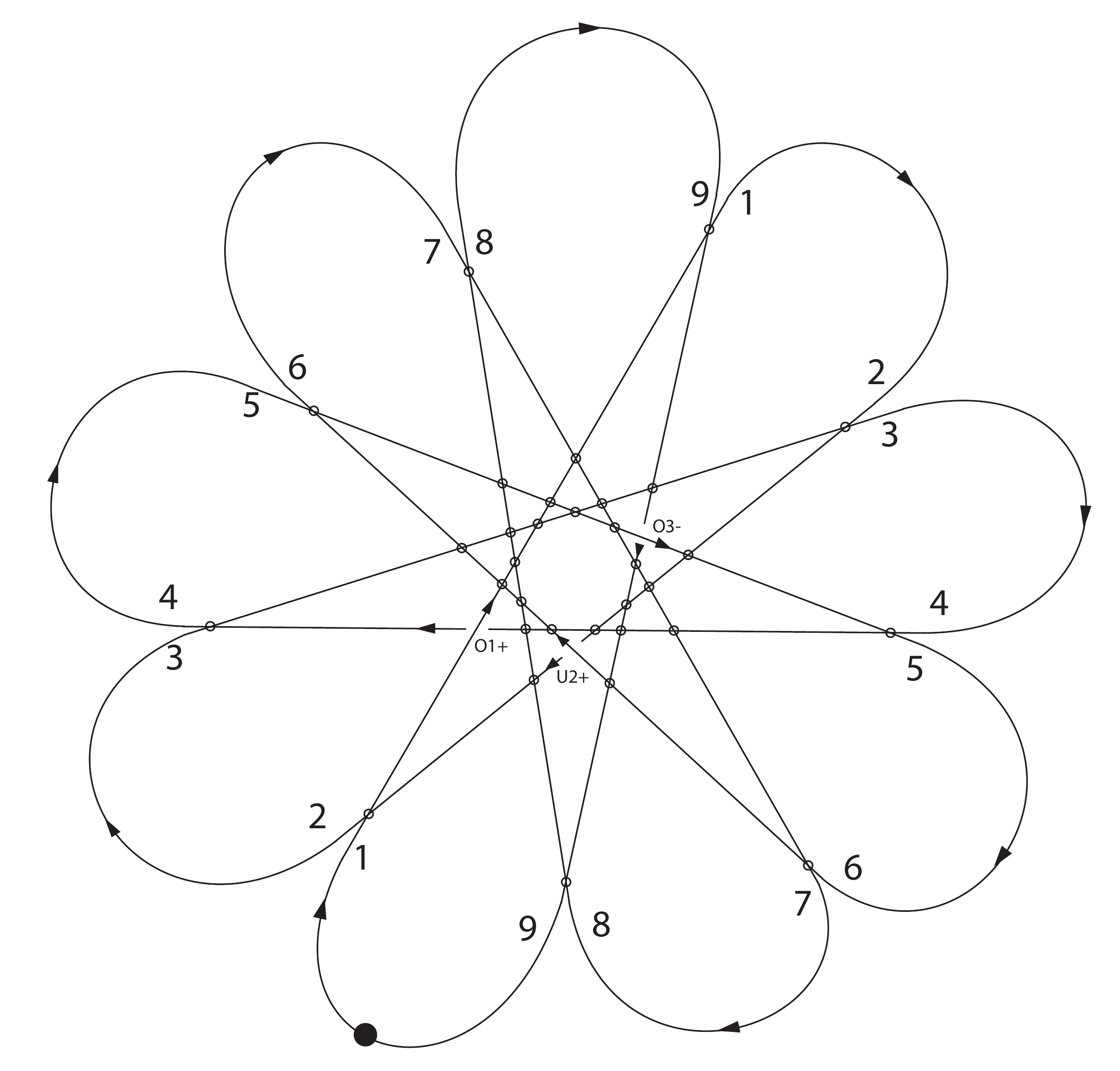}
\caption{Creating a petal diagram for signed Gauss code O1+ U2+ U1+ O3- O2+ U3-.}
\label{petalconstruction}
\end{figure}

\begin{definition} Given a virtual knot $K$, define its {\it virtual petal number} $p^V(K)$ to be the least number of petals in any virtual petal diagram for the knot. 
\end{definition} 

The algorithm does not necessarily generate the least virtual petal number, as is demonstrated by the virtual trefoil, for which the algorithm generates a virtual petal diagram with 7 petals, but in fact $p^V(VT) = 5$. However, it gives an upper bound that is linear in the number of classical crossings.

\begin{corollary}
Given a virtual knot $K$, $p^{V}(K) \leq 3c_2^{cl}(K)$ when $c_2^{cl}(K)$ is odd and $p^{V}(K) \leq 3c_2^{cl}(K) +1$ when $c_2^{cl}(K)$ is even.
\end{corollary}

In the case that the knot $K$ is classical, such a linear upper bound on the classical petal number $p(K)$ was established in~\cite{even2018distribution}, implying that also $p^V(K) \leq p(K) < 2c_2(K)$. This linear bound on the classical petal number cannot be improved in general, since $p(K) > c_2(K)$ for the infinite family of alternating knots~ \cite{Adams2}, but it may be possible that virtual petal diagrams are more efficient for classical knots.\\

\noindent {\bf Question 5:} Can a classical knot have virtual petal number less than its classical petal number?\\

\noindent {\bf Question 6:} Does there exist a linear lower bound of the form $p^{V}(K) \geq \alpha\,c_2^{cl}(K)$ for an infinite family of classical or virtual knots?\\

We remark that although our construction implies that a classical knot has a virtual petal diagram, it cannot replace the previous proof that a classical knot has a classical petal diagram~\cite{Adams2}. There are no moves known that would preserve the petal diagram while eliminating the virtual crossings.

\bibliography{references}{}
\nocite{*}
\bibliographystyle{plain}

\end{document}